\documentclass[12pt]{amsart}

\usepackage{amsmath}
\usepackage{amsfonts}
\usepackage{amssymb}
\usepackage[a4paper]{geometry}
\usepackage{amsthm}

\newtheorem{theorem}{Theorem}
\newtheorem{proposition}{Proposition}
\newtheorem{lemma}{Lemma}

\author[O. Ostrovska]{Olha Ostrovska}
\email{olyushka.ostrovska@gmail.com}
\address{National Technical University of Ukraine ``Igor Sikorsky Kyiv Polytechnic Institute'')}

\author[V. Ostrovskyi]{Vasyl Ostrovskyi}
\email{vo@imath.kiev.ua}
\address{Institute of Mathematics, NAS of Ukraine}

\author[D. Proskurin]{Danylo Proskurin}
\email{prohor75@gmail.com}
\address{Kyiv National Taras Shevchenko University}

\author[Yu. Samoilenko]{Yurii Samoilenko}
\email{yurii\_sam@imath.kiev.ua}
\address{Institute of Mathematics, NAS of Ukraine}

\title[Representations of $q_{ij}$-commuting isometries]{A class of representations of $C^*$-algebra generated by $q_{ij}$-commuting isometries}
\date{}

\begin{document}
\begin{abstract}
 For $C^*$-algebra generated by a finite family of isometries $s_j$, $j=1,\dots,d$ satisfying $q_{ij}$-commutation relations
 \[
  s_j^* s_j = I, \quad s_j^* s_k = q_{ij}s_ks_j^*, \qquad q_{ij} = \bar q_{ji}, |q_{ij}|<1, \ 1\le i \ne j \le d, 
 \]
we construct an infinite  family of unitarily non-equivalent irreducible representations. These representations are deformations of the corresponding class of representations of the Cuntz algebra $\mathcal O_d$.
\end{abstract}
\maketitle

\section{Introduction}
Algebras with Wick ordering \cite{pos,jsw,wick3} and their representations have been studied by various authors. Wick algebra $W_d(T)$ is an associative unital $*$-algebra over $\mathbb C$ generated by a finite number of elements $a_j, a_j^*$, $j=1.\dots, d$, for which the following relations hold
\[
 a_j^*a_k = \delta_{jk} I + \sum_{l,m=1}^d T_{jk}^{lm}a_ma_l^*, \quad \bar T_{jk}^{lm} = T_{kj}^{ml}, \ j,k,l,m =1,\dots, d.
\]
In particular, if all $T_{jk}^{lm}=0$, then relations in $W_d(T)$ generate the Cuntz-Toeplitz $C^*$-algebra $\mathcal O_d^0$,
\[
 a_j^* a_k = \delta_{jk} I, \quad j,k =1,\dots, d.
\]

It was cojectured (and proved in some cases) in \cite{wick3} that for smal enough coefficients $T_{jk}^{lm}$, $j,k,l,m=1,\dots, d$, for $W_d(T)$ there exists a universal enveloping $C^*$-algebra which is isomorphic to $\mathcal O_d^0$. In \cite{jps2}, it has been shown that the $C^*$-algebra generated by a pair of $q$-commuting isometries,
\[
 s_1^*s_1 = s_2^*s_2 =I, \quad s_1^*s_2 = q s_2s_1^*,
\]
is isomorphic to $\mathcal O_2^0$ for all $|q|<1$; however, in the case of $d>2$, such result is still a conjecture.

In this paper, we study representations of Wick ordered $C^*$-algebra $W$ generated by elements $s_j$, $j=1,\dots, d$, satisfying relations
\begin{equation}\label{eq:q_ij}
 s_i^*s_i = I, \quad s_i^*s_j = q_{ij}s_js_i^*, \quad |q_{ij}|<1, \ q_{ij} = \bar q_{ji}, \qquad 1\le i\ne j\le d. 
\end{equation}
In general situation of Wick algebras the most known is the Fock representation \cite{fock1,fock2} and coherent \cite{coherent} representations. We construct a new family of irreducible non-Fock representations which are deformations of a certain family of irreducible representations of the Cuntz algebra $\mathcal O_d$.

\section{Preliminaries}
\subsection{Some notations}
We start with some notations. Let $\alpha =(\alpha_1,\dots,\alpha_m)\in \{1,\dots,d\}^m$ be a finite multiindex of length $m$, $|\alpha|=m$, let $\Lambda_m = \{1,\dots,d\}^m$ be the set of all finite multiindices of length $m$, $\Lambda_0=\emptyset$, and let $\Lambda^0 =\cup_{m=0}^\infty \Lambda_m$ be the set of all finite multiindices of arbitrary length. Also, we will use the set $\Lambda = \{1,\dots,d\}^\infty$ of all infinite multiindices.
For each finite multiindex $\alpha =(\alpha_1,\dots,\alpha_m) \in \Lambda^0$ we use notation $s_\alpha = s_{\alpha_1}\dots s_{\alpha_m}$. For a finite multiindex we use standard mappings:
\begin{align*}
 \Lambda_m \ni \alpha &= (\alpha_1,\dots,\alpha_m) \mapsto \sigma (\alpha) = (\alpha_2,\dots,\alpha_m) \in \Lambda_{m-1},
 \\
 \Lambda_m \ni \alpha &= (\alpha_1,\dots,\alpha_m) \mapsto \sigma_j (\alpha) = (j,\alpha_2,\dots,\alpha_m) \in \Lambda_{m+1}, \quad j=1,\dots, d.
\end{align*}

If $\alpha$ does not contain $j$, then \eqref{eq:q_ij} implies
\[
 s_j^*  s_\alpha = q(j,\alpha) s_\alpha s_j^*, \quad q(j,\alpha) = q_{j\alpha_1}\dots q_{j\alpha_m}.
\]
If $\alpha$ contains $j$, then $\alpha$ can be represented as $\alpha=(\alpha'j\alpha'')$, where $\alpha'$ does not contain $j$, then
\[
 s_j^*s_\alpha = q(j,\alpha')s_{\alpha'}s_{\alpha''}  = q(j,\alpha)s_{\alpha\setminus j}
\]
(here and below, we denote by $\alpha\setminus j = (\alpha'\alpha'')$ multiindex obtained from $\alpha$ by removing the first occurrence of $j$, and set $q(j,\alpha) = q(j,\alpha')$ for convenience).

Similarly, for any finite multiindices $\alpha = (\alpha_1,\dots, \alpha_m)\in \{1,\dots,d\}^m$, $\beta = (\beta_1,\dots,\beta_n)\in \{1,\dots,d\}^n$, $n,m\ge 1$, one can define $\alpha \setminus \beta$ inductively as follows: 
\[
 \alpha\setminus \beta = \begin{cases}(\alpha\setminus \beta_1) \setminus \sigma(\beta), & \beta_1 \subset \alpha,\\ \alpha \setminus \sigma(\beta), & \text{otherwise}\end{cases}
\]
Setting  $s_\emptyset =I$, we obtain 
\[
 s_\alpha^* s_\beta = q(\alpha,\beta) s_{\beta\setminus \alpha} s^*_{\alpha\setminus \beta},
\]
where $q(\alpha,\beta)$ is calculated in an obvious way inductively.

If $\beta$ is a permutation of $\alpha$, then $\alpha \setminus \beta = \beta\setminus \alpha = \emptyset$, and we have $s_\alpha^* s_\beta = q(\alpha,\beta) I$. Also, in this case, for any multiindex $\delta$ we have 
\[
 s_{(\alpha\delta)}^* s_{(\beta\delta)} = q(\alpha,\beta)I = q((\alpha\delta),(\beta\delta)) I.
\]
For infinite multiindices, we define
\begin{equation}\label{eq:qbg}
 q(\alpha,\beta) = \lim_{m\to\infty} q((\alpha_1,\dots,\alpha_m),(\beta_1,\dots,\beta_m)).
\end{equation}
The limit exists since $|q|<1$ and is nonzero only if the sequence becomes stationary, i.e., if there exists $m$ such that $\sigma^m(\alpha) = \sigma^m(\beta)$, and $(\beta_1,\dots,\beta_m)$ is a permutation of $(\alpha_1,\dots, \alpha_m)$.

\subsection{Fock representation}

Fock representation $\pi_F$ is a $*$-representation of $W$ which is determined by the condition that there exists a (unique up to a constant) unit vector $\Omega$ for which $\pi_F(s_j^*)\Omega =0$, $1\le j\le d$. The representation space $\mathcal F$ (Fock space) is spanned by vectors $e_\alpha = \pi_F(s_\alpha) \Omega$, $\alpha \in \Lambda^0$ equipped with the Fock scalar product
\[
 (e_\alpha,e_\beta)_F = (\pi_F(s_\beta^*s_a\alpha)\Omega,\Omega)_F
 =\begin{cases} q(\beta,\alpha),&|\alpha|=|\beta|, \ \text{$\beta$ is a permutation of $\alpha$,}\\ 0, &\text{otherwise.}\end{cases}
\]
This scalar product is known to be positive \cite{pos}. 
In particular, the finite-dimensional subspaces $\mathcal F_n$ spanned by $e_\alpha$, $\alpha \in \Lambda_n$,  $n\ge 0$, are orthogonal to each other, therefore,
\[
 \mathcal F = \bigoplus_{n=0}^\infty \mathcal F_n.
\]

Operators of the Fock representation are defined as follows
\begin{align*}
 \pi_F(s_j) e_\alpha &= e_{\sigma_j(\alpha)}, 
 \\
 \pi_F(s_j^*) e_\alpha &= q(j,\alpha) \pi_F (s_{\alpha\setminus j}) \pi_F (s_{j\setminus \alpha}^*)\Omega = \begin{cases} q(j,\alpha) e_{\alpha \setminus j},& \text {$\alpha$ contains $j$}, \\ 0& \text{otherwise.}\end{cases}
\end{align*}

\section{A construction of non-Fock representations}

We start with introducing an appropriate Hilbert space. Consider an (uncountable) set of vectors $e_\gamma$, $\gamma \in \Lambda$. For these vectors, define
\begin{equation} \label{eq:scalar_product}
 (e_\beta,e_\gamma) = q(\gamma,\beta),
\end{equation}
where $q(\beta,\gamma)$ was defined above in \eqref{eq:qbg}. In particular, for any $\beta \in \Lambda$ we have $(e_\beta,e_\beta)=1$.

We say that infinite multiindices $\alpha,\beta \in \Lambda$ are equivalent, denoted by $\beta\sim\alpha$, if they ``have the same tails up to a shift'', i.e., there exist numbers $m,n$, such that $\sigma^m(\beta) = \sigma^n(\gamma)$. Fix an infinite multiindex $\alpha$ and consider a countable family of vectors $(e_\beta \mid \beta \sim\alpha)$. Define $\tilde H_\alpha$ as a linear span of this family.

\begin{proposition}
 Form \eqref{eq:scalar_product} is well-defined and positive on $\tilde H_\alpha$.
\end{proposition}

\begin{proof}
Fix a sequence $\Lambda \ni \alpha =(\alpha_1,\alpha_2,\dots)$, and
define operators $J_k\colon \mathcal F_k \to \mathcal F _{k+1}$, $k=0,1,\dots$ as follows.
\[
 \mathcal F_k \ni e_{(\gamma_1,\dots,\gamma_k)} \mapsto J_k e_{(\gamma_1,\dots,\gamma_k)} = e_{(\gamma_1,\dots,\gamma_k, \alpha_{k+1})} \in \mathcal F_{k+1}, \quad k =0,1,\dots,
\]
and extend this action to the whole $\mathcal F_{k}$ by linearity. These operators are well-defined since $(e_\gamma)_{\gamma \in \Lambda_k}$ form a linear basis in $\mathcal F_k$.

\begin{lemma}\label{lm:embed}
 Operator $J_k$, $k\ge 0$, is an isometric embedding of $\mathcal F_k$ into $\mathcal F_{k+1}$.
\end{lemma}

\begin{proof}
 Take $\beta,\gamma \in \Lambda _{k}$, then
 \begin{align*}
  (J_k e_\beta, J_k e_\gamma )_{\mathcal F_{k+1}}& = (e_{(\beta_1,\dots,\beta_k,\alpha_{k+1})},e_{(\gamma_1,\dots,\gamma_k,\alpha_{k+1})})
  \\
  & = (\pi_F(s_\gamma^*s_\beta)e_{\alpha_{k+1}},e_{\alpha_{k+1}}) 
  = q(\gamma,\beta) (\pi_F(s_{\beta\setminus\gamma}s_{\gamma\setminus \beta}^*)e_{\alpha_{k+1}},e_{\alpha_{k+1}})
  \\
  &= q(\gamma,\beta) (\pi_F(s_{\gamma\setminus \beta})^*e_{\alpha_{k+1}},\pi_F(s_{\beta\setminus\gamma})^*e_{\alpha_{k+1}}) = (e_\beta,e_\gamma)_{\mathcal F_k}.
 \end{align*}
 Indeed, if $\gamma$ is a permutation of $\beta$, then $\beta\setminus\gamma = \gamma\setminus\beta=\emptyset$, and since $s_{\emptyset} =I$, 
 \begin{gather*}
  q(\gamma,\beta) (\pi_F(s_{\gamma\setminus \beta})^*e_{\alpha_{k+1}},\pi_F(s_{\beta\setminus\gamma})^*e_{\alpha_{k+1}}) 
  = q(\gamma,\beta) (e_{\alpha_{k+1}},e_{\alpha_{k+1}}) 
  \\
  =  q(\gamma,\beta)  = (e_\beta,e_\gamma)_{\mathcal F_k}.
 \end{gather*}
If $\gamma$ is not a permutation of $\beta$, then $\beta\setminus\gamma$ and  $\gamma\setminus\beta$ are non-empty and disjoint. Since $\pi_F(s_j)^* e_{\alpha_{k+1}} \ne 0$ only if $j=\alpha_{k+1}$, this implies that 
\[
 q(\gamma,\beta) (\pi_F(s_{\gamma\setminus \beta})^*e_{\alpha_{k+1}},\pi_F(s_{\beta\setminus\gamma})^*e_{\alpha_{k+1}}) =0 = (e_\beta,e_\gamma)_{\mathcal F_k}. \qedhere
\]
\end{proof}

Consider an inductive limit $\tilde H^0_\alpha =\varinjlim _{J_k} \mathcal F_{k}$. This space can be naturally identified with a span of vectors $e_\beta$, $\beta \in\Lambda$ over all $\beta$ for which there exists $m$ such that $\beta_k = \alpha_k$ for $k > m$. 
Lemma~\ref{lm:embed} and the positivity of the Fock scalar product yields that
\[
 (e_\beta,e_\gamma) = (e_{(\beta_1,\dots,\beta_k)},e_{(\gamma_1,\dots,\gamma_k)}), \quad \text{if $\beta_j=\gamma_j =\alpha_j$ for all $j>k$}
\]
is a well-defined positive form on $\tilde H_\alpha^0$, and that 
\[
 \mathcal F_k \ni e_{(\gamma_1,\dots, \gamma_k)} \mapsto e_{(\gamma_1,\dots,\gamma_k,\alpha_{k+1},\alpha_{k+2},\dots)} \in \tilde H_\alpha^0
\]
is an isometric embedding.

Obviously, $\tilde H^0_\alpha \subset \tilde H_\alpha$, but in general (unless $\alpha$ is equivalent to a stationary sequence) it is a proper subset. Let $\beta \sim\alpha$. The same way, consider space $\tilde H^0_\beta$ with the corresponding scalar product. If there exists $m$ such that $\beta_k=\alpha_k$ for all $k>m$ (or equivalently, $\sigma^m(\alpha) = \sigma^m(\beta)$), then $\tilde H_\alpha^0 = \tilde H^0_\beta$. Otherwise $\alpha$ and $\beta$ differ in an infinite set of indices (so that $q(\alpha,\beta) =0$) and we set $\tilde H_\alpha^0$ and  $\tilde H^0_\beta$ to be orthogonal. Similarly, for any $\beta \sim\gamma\sim\alpha$ we define 
\[
 (e_\beta,e_\gamma) = 
 \begin{cases}
 (e_{(\beta_1,\dots,\beta_m)},e_{(\gamma_1,\dots,\gamma_m)})_{\mathcal F_{m}},& \text{there exists $m$, for which $\sigma^m(\beta) = \sigma^m(\gamma)$},
 \\
 0, &\text{otherwise.}
 \end{cases}
\]
The arguments above show that this form is well-defined and positive on the whole $\tilde H_\alpha$. One can easily see that the latter expression is equal to $q(\gamma,\beta)$.
\end{proof}

Define $H_\alpha$ as a completion of $\tilde H_\alpha$ with respect to the introduced above scalar product.

\begin{theorem}
1. Operators in $H_\alpha$
 \begin{align*}
  \pi_\alpha(s_j) e_\beta = e_{\sigma_j(\beta)}, \quad 
  \pi_\alpha(s_j^*) e_\beta = \begin{cases}0, & \text{$\beta$ does not contain $j$,} \\q(j,\beta)e_{\beta\setminus j}, & \text{otherwise,}\end{cases}
 \end{align*}
form well-defined $*$-representation of the $C^*$-algebra $W$.
  
 2. This representation is irreducible
 
 3. Representations corresponding to multiindices $\alpha,\alpha'$ are unitary equivalent iff the corresponding Hilbert spaces coincide, i.e., $\alpha\sim\alpha'$. 

 4. Representation $\pi_\alpha$ is not unitary equivalent to the Fock representation.
\end{theorem}

\begin{proof}
1. We need to verify that $\pi_\alpha(s_j^*) = \pi_\alpha(s_j)^*$, $j=1,\dots, d$, and that relations \eqref{eq:q_ij} hold. Obviously, it is sufficient to verify this on the vectors $e_\gamma$, $\gamma \sim\alpha$. 

Conditions  $\pi_\alpha(s_j^*)e_\beta = \pi_\alpha(s_j)^*e_\beta$, $j=1,\dots, d$,  holds due to the way the scalar product is constructed. Indeed,  
\begin{align*}
 (\pi_\alpha(s_j)^* e_\beta,e_\gamma) & =  ( e_\beta, \pi_\alpha(s_j)e_\gamma) = ( e_\beta, e_{\sigma_j(\gamma)}),
 \\
 (\pi_\alpha(s_j^*) e_\beta,e_\gamma) & = 
 \begin{cases}0, & \text{$\beta$ does not contain $j$,} \\q(j,\beta)(e_{\beta\setminus j},e_\gamma), & \text{otherwise.}\end{cases}
 \end{align*}
If $\beta$ does not contain $j$, then the both expression are zero. Assume $\beta$ contains $j$.
According to the definition of the scalar product, $( e_{\beta\setminus j},e_\gamma) \ne 0$ only if there exists $m_1$, for which $\sigma^{m_1}(\beta \setminus j) = \sigma^{m_1}(\gamma)$. Then there exists $m_2$, for which $\sigma^{m_2}(\beta) = \sigma^{m_2}(\sigma_j(\gamma))$. Take $m= \max (m_1,m_2)$, this will in particular ensure that $(\beta_1,\dots,\beta_{m+1})$ contains $j$. Then, since the Fock representation is a $*$-representation,
\begin{align*}
 ( e_\beta, e_{\sigma_j(\gamma)})& = ( e_{(\beta_1,\dots,\beta_{m+1})}, e_{(j,\gamma_1,\dots,\gamma_m)})_{\mathcal F_{m+1}} 
 \\&
 = ( e_{(\beta_1,\dots,\beta_{m+1})},\pi_F(s_j) e_{(\gamma_1,\dots,\gamma_m)})_{\mathcal F_{m+1}}
 \\&
 =(\pi_F(s_j^*) e_{(\beta_1,\dots,\beta_{m+1})}, e_{(\gamma_1,\dots,\gamma_m)})_{\mathcal F_{m}}
 \\&
 =q(j,\beta)(e_{(\beta_1,\dots,\beta_{m+1}) \setminus j}, e_{(\gamma_1,\dots,\gamma_m)})_{\mathcal F_{m}} = q(j,\beta)( e_{\beta\setminus j},e_\gamma).
\end{align*}

To prove \eqref{eq:q_ij}, we apply the same arguments as above to reduce the situation to the case of the Fock representation.

2. We start with the following auxiliary fact.

\begin{lemma}\label{lm:biort}
 In the $C^*$-algebra $W$ there exist elements $\tilde s_j$, such that $\tilde s_j^* s_k = \delta_{jk} I$.
\end{lemma}

\begin{proof}
For each $j=1,\dots, d$, let $p_j = s_js_j^*$ be projection on the range of $\pi_\alpha(s_j)$. Since the $C^*$-algebra generated by $p_j$, $j=1,\dots,d$, is finite-dimensional \cite{ku-po}, the latter $C^*$-algebra, and therefore, $W$ as well, contains $\check p_j=\bigvee_{k\ne j}p_k$ which is a projection on the sum of ranges of $\pi_\alpha(s_k)$, $k\ne j$.

Write $c_j=(I-\check p_j)s_j$. Then $c_j^*s_k=0$, $k\ne j$, and if we show that  $c_j^*s_j$ is invertible, then 
\[
\tilde s_j =  (I-\check p_j)s_j^*(s_j^*(I-\check p_j)s_j)^{-1}
\]
is the necessary element.

So it is sufficient to prove that $c_j^*s_j = s_j^*(I - \check p_j)s_j$ is invertible, First notice that for any element $x=s_\mu s^*_\nu$, where $\mu$ and $\nu$ do not contain $j$, we have
\[
 s_j^* xs_j = s_j^* s_\mu s^*_\nu s_j = q(j,\mu) q(\nu,j)   s_\mu s_j^*s_j s^*_\nu =  q(j,\mu) q(\nu,j) x,
\]
therefore,  $s_j^*(I - \check p_j)s_j$ belongs to the finite-dimensional algebra generated by $p_k$, $k\ne j$, and thus its spectrum is finite. Then to prove the invertibility of  $s_j^*(I - \check p_j)s_j$, it is enough to show that it has zero kernel, which is equivalent to $\ker (I-\check p_j)s_j =0$.

Since the Fock representation of $W$ is exact \cite{fock1}, any element $x\in W$ can be uniquely represented as
\[
 x = p_1x_1+\dots+p_dx_d.
\]
This means that the range of $p_j$ is linearly independent of the span of the ranges of $p_k$, $k\ne j$, in particular, the ranges of $p_j$ and $\check p_j$ do not intersect. This implies  $\ker (I-\check p_j)s_j =0$.
\end{proof}

Now we prove the irreducibility of $\pi_\alpha$. For $\mu = (\mu_1,\dots,\mu_n) \in\Lambda_n$, denote $\tilde s_\mu = \tilde s_{\mu_1}\dots\tilde s_{\mu_n}$. Fix an infinite sequence $\mu=(\mu_1,\mu_2,\dots) \in \Lambda$ and consider operators
\[
 P_n(\mu) = \pi_\alpha(s_{(\mu_1,\dots,\mu_n)} \tilde s^*_{(\mu_1,\dots,\mu _n)}, \quad n\ge 1.
\]
For any vector of the form $e_\beta$, one directly see that 
\[
 \lim_{n\to\infty} P_n(\mu) e_{\beta} = \delta_{\beta\mu} e_\mu,
\]
i.e., the sequence $P_n(\mu)$ strongly converges to $P(\mu)$ which is a projection onto the one-dimensional space generated by $e_\mu$. 

Let $C$ be a bounded operator commuting with all $\pi_\alpha(x)$, $x\in W$. Then $C$ commutes with $P(\mu)$, and therefore, for any $\beta\sim\alpha$ we have $Ce_\beta = c(\beta)e_\beta$, where $c(\beta)$ is a constant. On the other hand, by the construction of $\tilde s_j$ given by Lemma~\ref{lm:biort}, for each $\beta,\gamma \sim\alpha$ there exist finite multiindices $\mu,\nu$, such that $s_\gamma = s_\mu \tilde s^*_\nu s_\beta$, so that $\pi_\alpha(s_\mu \tilde s^*_\nu) e_\beta = e_\gamma$. Since $C$ commutes with $\pi_\alpha(s_\mu \tilde s^*_\nu)$, we have
\[
 c(\beta)  e_\gamma = c(\beta) \pi_\alpha(s_\mu \tilde s^*_\nu) e_\beta 
 = \pi_\alpha(s_\mu \tilde s^*_\nu) C e_\beta 
 = C \pi_\alpha(s_\mu \tilde s^*_\nu)  e_\beta
 = C  e_\gamma = c(\gamma)e_\gamma, 
\]
i.e., $c(\beta) = c(\gamma)$ for all $\beta,\gamma\sim\alpha$. Therefore, $C$ is a scalar operator, and by the Schur lemma, $\pi_\alpha$ is irreducible.

3. Using the arguments above, if $\alpha'$ is not equivalent to $\alpha$, then 
\[
 \lim_{n\to\infty} P_n(\alpha') e_\beta =0, \quad \beta\sim\alpha,
\]
therefore, $P_n(\alpha')$ strongly converges to zero in $H_\alpha$, but in $H_{\alpha'}$ it strongly converges to a nonzero projection $P(\alpha')$.

4. Assume that there exists $\Omega \in H_\alpha$, for which $\pi_\alpha (s_j)^* \Omega =0$, $j=1,\dots, d$. For each $\beta\sim\alpha$ we have
\[
 (\Omega, e_\beta) = (\Omega, \pi_\alpha(s_{\beta_1})e_{\sigma(\beta)}) = (\pi_\alpha(s_{\beta_1})^*\Omega, e_{\sigma(\beta)}) =0.
\]
Since vectors $e_{\beta}$, $\beta \sim \alpha$, form a total set in $H_\alpha$, we get $\Omega =0$.
\end{proof}

\section{Acknowledgement}
The authors express their gratitude to Kostyantyn Krutoi for helpful discussions of the subject of this paper.

\end{document}